\documentclass[reqno]{amsart}
\usepackage{hyperref}
\def\N{{\mathbb{N}}}

\def\R{{\mathbb{R}}}

\begin{document}
\title[Infinite dimension]{Infinite dimensional multipliers and Pontryagin principles for discrete-time problems}

\author[Bachir and Blot]
{Mohammed Bachir and Jo${\rm \ddot e}$l Blot} 

\address{Mohammmed Bachir: Laboratoire SAMM EA4543,\newline
Universit\'{e} Paris 1 Panth\'{e}on-Sorbonne, centre P.M.F.,\newline
90 rue de Tolbiac, 75634 Paris cedex 13,
France.}
\email{Mohammed.Bachir@univ-paris1.fr}

\address{Jo\"{e}l Blot: Laboratoire SAMM EA4543,\newline
Universit\'{e} Paris 1 Panth\'{e}on-Sorbonne, centre P.M.F.,\newline
90 rue de Tolbiac, 75634 Paris cedex 13,
France.}
\email{blot@univ-paris1.fr}
\date{June 28th 2016}

\numberwithin{equation}{section}
\newtheorem{theorem}{Theorem}[section]
\newtheorem{lemma}[theorem]{Lemma}
\newtheorem{example}[theorem]{Example}
\newtheorem{remark}[theorem]{Remark}
\newtheorem{definition}[theorem]{Definition}
\newtheorem{proposition}[theorem]{Proposition}
\newtheorem{corollary}[theorem]{Corollary}
\begin{abstract}
The aim of this paper is to provide improvments to Pontryagin principles in infinite-horizon discrete-time framework when the space of states and of space of controls are infinite-dimensional. We use the method of reduction to finite horizon and several functional-analytic lemmas to realize our aim. 
\end{abstract}
\maketitle
M.S.C. 2010: 49J21, 65K05, 39A99, 46B99, 54E52.
\section{Introduction}
We treat of problems of Optimal Control in infinite horizon which are governed by discrete-time controlled dynamical systems in the following forms
\begin{equation}\label{eq11}
x_{t+1} = f_t(x_t,u_t), \hskip2mm t \in \N,
\end{equation}
or
\begin{equation}\label{eq12}
x_{t+1} \leq f_t(x_t,u_t), \hskip2mm t \in \N.
\end{equation}
$X$ and $U$ are real Banach spaces, the state variable is $x_t \in X_t \subset X$, the control variable is $u_t \in U_t \subset U$ and $f_t : X_t \times U_t \rightarrow X_{t+1}$ is a mapping. For \eqref{eq12}, $X$ is endowed with a structure of ordered Banach space, and  its positive cone, $X_+ := \{ x \in X : x \geq 0 \}$, is closed, convex and satisfies $X_+ \cap (- X_+) = \{ 0 \}$. An initial state, $\sigma \in X_0$ is fixed. We define ${\rm Ead}(\sigma)$ the set of the processes $((x_t)_{t\in \N}, (u_t)_{t \in \N})$ which belong to $\prod_{t \in \N} X_t \times \prod_{t \in \N} U_t$, which satisfy $x_0 = \sigma$ and \eqref{eq11} for all $t \in \N$.  We define ${\rm Iad}(\sigma)$ the set of the processes $((x_t)_{t\in \N}, (u_t)_{t \in \N})$ which belong to $\prod_{t \in \N} X_t \times \prod_{t \in \N} U_t$, which satisfy $x_0 = \sigma$ and \eqref{eq12} for all $t \in \N$.
\vskip1mm
To define criteria, we consider functions $\phi_t : X_t \times U_t \rightarrow \R$ for all $t \in \N$. From these functions we build the functional
\begin{equation}\label{eq13}
J((x_t)_{t\in \N}, (u_t)_{t \in \N}) := \sum_{t = 0}^{+ \infty} \phi_t(x_t,u_t).
\end{equation}
Notice that $J$ is not defined for all the processes $((x_t)_{t\in \N}, (u_t)_{t \in \N}) \in \prod_{t \in \N}X_t \times \prod_{t \in \N} U_t$. 
And so we define the set of the processes $((x_t)_{t\in \N}, (u_t)_{t \in \N}) \in {\rm Ead}(\sigma)$ (respectively in ${\rm Iad}(\sigma)$) such that the series $\sum_{t = 0}^{+ \infty} \phi_t(x_t,u_t)$ converges into $\R$, and we denote this set by ${\rm Edom}(\sigma)$ (respectively ${\rm Idom}(\sigma)$).
\vskip3mm
We also consider other criteria to define the following problems.
\vskip1mm
\noindent
$({\bf PE}_1(\sigma))$ : Find $((\hat{x}_t)_{t\in \N}, (\hat{u}_t)_{t \in \N}) \in {\rm Edom}(\sigma)$ such that $J((\hat{x}_t)_{t\in \N}, (\hat{u}_t)_{t \in \N}) \geq J((x_t)_{t\in \N}, (u_t)_{t \in \N})$ for all $((x_t)_{t\in \N}, (u_t)_{t \in \N}) \in {\rm Edom}(\sigma)$.
\vskip1mm
\noindent
$({\bf PE}_2(\sigma))$ : Find $((\hat{x}_t)_{t\in \N}, (\hat{u}_t)_{t \in \N}) \in {\rm Ead}(\sigma)$ such that \\$\limsup_{T \rightarrow + \infty} \sum_{t=0}^T (\phi_t(\hat{x}_t,\hat{u}_t) - \phi_t(x_t,u_t)) \geq 0$ for all $((x_t)_{t\in \N}, (u_t)_{t \in \N}) \in {\rm Ead}(\sigma)$.
\vskip1mm
\noindent
$({\bf PE}_3(\sigma))$ : Find $((\hat{x}_t)_{t\in \N}, (\hat{u}_t)_{t \in \N}) \in {\rm Ead}(\sigma)$ such that \\$\liminf_{T \rightarrow + \infty} \sum_{t=0}^T (\phi_t(\hat{x}_t,\hat{u}_t) - \phi_t(x_t,u_t) ) \geq 0$ for all $((x_t)_{t\in \N}, (u_t)_{t \in \N}) \in {\rm Ead}(\sigma)$.
We also consider similar problems when the system is governed by \eqref{eq12} instead of \eqref{eq11}.
\vskip1mm
\noindent
$({\bf PI}_1(\sigma))$ : Find $((\hat{x}_t)_{t\in \N}, (\hat{u}_t)_{t \in \N}) \in {\rm Idom}(\sigma)$ such that \\
$J((\hat{x}_t)_{t\in \N}, (\hat{u}_t)_{t \in \N}) \geq J((x_t)_{t\in \N}, (u_t)_{t \in \N})$ for all $((x_t)_{t\in \N}, (u_t)_{t \in \N}) \in {\rm Idom}(\sigma)$.
\vskip1mm
\noindent
$({\bf PI}_2(\sigma))$ : Find $((\hat{x}_t)_{t\in \N}, (\hat{u}_t)_{t \in \N}) \in {\rm Iad}(\sigma)$ such that \\$\limsup_{T \rightarrow + \infty} \sum_{t=0}^T (\phi_t(\hat{x}_t,\hat{u}_t) - \phi_t(x_t,u_t)) \geq 0$ for all $((x_t)_{t\in \N}, (u_t)_{t \in \N}) \in {\rm Iad}(\sigma)$.
\vskip1mm
\noindent
$({\bf PI}_3(\sigma))$ : Find $((\hat{x}_t)_{t\in \N}, (\hat{u}_t)_{t \in \N}) \in {\rm Iad}(\sigma)$ such that \\$\liminf_{T \rightarrow + \infty} \sum_{t=0}^T (\phi_t(\hat{x}_t,\hat{u}_t) - \phi_t(x_t,u_t)) \geq 0$ for all $((x_t)_{t\in \N}, (u_t)_{t \in \N}) \in {\rm Iad}(\sigma)$.
\vskip1mm
\noindent
These problems are classical in the theory of infinite-horizon discrete-time Optimal Control, \cite{Mi}, \cite{JBNH}.
\vskip1mm
A way to establish necessary optimality conditions in the form of Pontryagin principles for the above-mentionned problems is the method of reduction to finite horizon which appears in \cite{JBHC}. In \cite{JBNH} several variations of this method are given in the setting of the finite dimension, and in \cite{BB} we find the use of this method in the setting of the infinite dimension. The basic idea of this method is that when $((\hat{x}_t)_{t\in \N}, (\hat{u}_t)_{t \in \N})$ is optimal for one of the previous problems, its restriction to $\{0,..., T \}$ is an optimal solution of a finite-horizon optimization problem. Using on these finite-horizon optimization problems a Karush-Kuhn-Tucker theorem or a Multipliers Rule, we obtain multipliers indexed by the finite horizon $T$. The second step is to build, from these multipliers sequences, multipliers which are suitable for the infinite-horizon problems.
\vskip2mm
When $X$ and $U$ have an infinite dimension, several difficulties arise, notably due to the closure of the ranges of linear operators, and due to the fact that in infinite dimensional dual Banach space, the origine is contained in the weak-star closure of  its unit sphere.
\vskip2mm
Now we briefly describe the contents of the paper. In Section 2, we give the statements of the main results which are Pontryagin principles. In Section 3, we establish results of Functional Analysis which are useful in the sequel. In Section 4, we establish results on Lagrange and Karush-Khun-Tucker multipliers. In Section 5, we give the proofs of the results of section 2. In Section 6, we give some additional applications of results that we use in the proof of the main theorems.
\section{The pontryagin principles.}
First we specify some notation.\vskip2mm
\noindent
We denote by $Int(A)$ the {\it topological interior} of a set $A$ and by $\overline{A}$ its {\it closure}. When $Z$ is a vector normed space, $z \in Z$ and $r \in (0, + \infty)$, $B_Z(z,r)$ denotes the {\it closed ball} with $z$ as center and $r$ as ray. The set co$(A)$ (respectively $\overline{\rm co}(A)$) stands for the {\it convex hull} (respectively the {\it closed convex hull}) of a subset $A$ in $Z$, Aff$(A)$ stands for the {\it affine hull} of a subset $A$ in $Z$. The {\it relative interior} of $A \subset Z$, denoted $\textnormal{ri}(A)$, is the topological interior of $A$ in the topological subspace Aff$(A)$.
\vskip2mm 
In the paper, the assumptions that we will use in our results belong to the following list of conditions. 
\begin{enumerate}
\item[(A1)] For all $t \in \N$, $X_t$ is a nonempty open subsets of $X$, and $U_t$ is a nonempty subsets of $U$.
\item[(A2)] $X$ is separable.
\end{enumerate}
And when $((\hat{x}_t)_{t\in \N}, (\hat{u}_t)_{t \in \N})$ is a given process and when $\sigma \in X_0$ is given, we consider the following conditions.
\begin{enumerate}
\item[(A3)] For all $t \in \N$, $\phi_t$ is Fr\'echet differentiable at $(\hat{x}_t,\hat{u}_t)$, $f_t$ is continuously Fr\'echet differentiable at $(\hat{x}_t,\hat{u}_t)$.
\item[(A4)] for all $t \in \N_*$, 
%$Df_t(\hat{x}_t, \hat{u}_t)(X\times T_{U_t}(\hat{u}_t))=X$ 
we have $0 \in Int\left[Df_t(\hat{x}_t, \hat{u}_t)((X\times T_{U_t}(\hat{u}_t))\cap B_{X\times U})\right]$, where $B_{X\times U}$ denotes the closed unit ball of $X\times U$.
\item[(A5)] For all $t \in \N$, the range of $D_2f_t(\hat{x}_t,\hat{u}_t)$ is closed and its codimension (in $X$) is finite.
\item[(A6)] There exists $s \in \N$ such that $A_s:=D_2f_s(\hat{x}_s,\hat{u}_s)(T_{U_s}(\hat{u}_s))$ contains a closed convex subset $K$ with $\textnormal{ri}(K)\neq\emptyset$ and such that $\overline{\textnormal{Aff}(K)}$ is of finite codimension in $X$.
\end{enumerate}
Recall that $T_{U_t}(\hat{u}_t) := \overline{ \{\alpha (u_t - \hat{u}_t) : \alpha \in [0, + \infty), u_t \in U_t \} }$.
We have not assume that the sets $U_t$ are open, but when we speak of the differentiability of a mapping $f$ on $X_t \times U_t$ at $(\hat{x}_t,\hat{u}_t)$, the meaning is that there exists a differentiable function $\tilde{f}$ defined on an open neighborhood of $(\hat{x}_t,\hat{u}_t)$ which is equal to $f$ on the intersection of this neighborhood and 
$X_t \times U_t$. When we speak of tangent cone, we consider the case where $U_t$ is convex. 
\begin{remark}\label{rem21}
Note that the condition \noindent (A6) is satisfied and is included in (A5), whenever there exists an $s\in \N$ such that $T_{U_s}(\hat{u}_s)=X$, in particular, if $\hat{u}_s$ belongs to the interior of $U_s$. 
\end{remark}
The first main result concerns the problems governed by \eqref{eq11}.
\vskip2mm
\begin{theorem}\label{thm22} Let $((\hat{x}_t)_{t \in \N}, (\hat{u}_t)_{t \in \N})$ be an optimal process of $({\bf PE}_k(\sigma))$ when $k \in \{1,2,3 \}$. Under [(A1)-(A6)], we assume moreover that $U_t$ is convex for all $t \in \N$. Then, there exist $\lambda_0 \in \R$ and $(p_t)_{t\geq 1} \in (X^*)^{\N_*}$ such that the following conditions hold. 
\begin{enumerate}
\item[(1)] $(\lambda_0,p_t)\neq (0,0)$, for all $t \geq s$.
\item[(2)] $\lambda_0 \geq 0$.
\item[(3)] $p_t = p_{t+1} \circ D_1f_t(\hat{x}_t, \hat{u}_t) + \lambda_0.D_1 \phi_t(\hat{x}_t, \hat{u}_t)$  for all $t \in \N_*$.
\item[(4)] $\langle \lambda_0.D_2\phi_t(\hat{x}_t, \hat{u}_t) + p_{t+1} \circ D_2f_t(\hat{x}_t, \hat{u}_t), u_t - \hat{u}_t \rangle \leq 0$ for all $t \in \N$ and for all $u_t \in U_t$.
\end{enumerate}
\end{theorem}
The second main result concerns the problems governed by \eqref{eq12}.
\begin{theorem}\label{thm23} Let $((\hat{x}_t)_{t \in \N}, (\hat{u}_t)_{t \in \N})$ be an optimal process of $({\bf PI}_k(\sigma))$ when $k \in \{1,2,3 \}$. Under [(A1)-(A6)], we assume moreover that $U_t$ is convex for all $t \in \N$ and that $Int(X_+) \neq \emptyset$. Then, there exist $\lambda_0 \in \R$ and $(p_t)_{t\geq 1} \in (X^*)^{\N_*}$ such that the following conditions hold. 
\begin{enumerate}
\item[(1)] $(\lambda_0,p_t)\neq (0,0)$, for all $t \geq s$.
\item[(2)] $\lambda_0 \geq 0$, and $p_n \geq 0$ for all $n \in \N_*$.
\item[(3)] $p_t = p_{t+1} \circ D_1f_t(\hat{x}_t, \hat{u}_t) + \lambda_0.D_1 \phi_t(\hat{x}_t, \hat{u}_t)$  for all $t \in \N_*$.
\item[(4)] $\langle \lambda_0.D_2\phi_t(\hat{x}_t, \hat{u}_t) + p_{t+1} \circ D_2f_t(\hat{x}_t, \hat{u}_t), u_t - \hat{u}_t \rangle \leq 0$ for all $t \in \N$ and for all $u_t \in U_t$.
\end{enumerate}
\end{theorem}
The proofs of Theorem \ref{thm22} and Theorem \ref{thm23} are based on the following two ideas: the first one is the reduction to finite horizon given by Lemma \ref{lem41} and the second one is to find criteria ensuring that the multipliers are not trivial in the infinite horizon. This criteria will be given by Lemma \ref{lem33}. 
%
%%%%%%%%%%%%%%%%%%%%%%%%%%%%%%%%%%%%%%%%%%%%%%%%%%%%%%%%%%%%%
\section{Preliminary results of Functional Analysis}
It is well known from Josefson-Nissenzweig theorem, (see [\cite{Dies}, Chapter XII]) that in infinite dimensional Banach space $Z$, there always exists a sequence $(p_n)_n$ in the dual space $Z^*$ that is weak$^*$ null and $\inf_{n \in \N} \|p_n\| >0$. In this section, we look about reasonable and usable conditions on a sequence of norm one in $Z^*$ such that this sequence does not converge to the origin in the $w^*$-topology. This situation has the interest, when we are looking for nontrivial multipliers for optimization problems, and was encountered several times in the literature. See for example \cite{BB} and \cite{MM}. The key is Lemma \ref{lem33} which permits to provide a solution to this problem.
We split this section in two subsections. The first is devoted to establish an abstract result (Lemma \ref{lem33}) which permits to avoid the Josefson-Nissenzweig phenomenon. The second is devoted to the consequences of this abstract result which are useful for our optimal control problem. 
\vskip2mm
 We need the following classical result.
\begin{proposition} \label{prop31} Let $C$ be a convex subset of a normed vector space. Let $x_0\in Int(C)$ and $x_1\in \overline{C}$. Then, for all $\alpha\in (0,1]$, we have $\alpha x_0+(1-\alpha)x_1\in Int(C)$. 
\end{proposition} 
We deduce the following useful proposition.
\begin{proposition} \label{prop32} Let $(F,\|.\|_F)$ be a normed vector space and $C$ be a closed convex subset of $F$ with non empty interior. Suppose that $D\subset C$ is a closed subset of $C$ with no empty interior in $(C,\|.\|_F)$ (for the topology induced by $C$). Then, the interior of $D$ is non empty in $(F,\|.\|_F)$. 
\end{proposition} 
\begin{proof} On one hand, there exists $x_0$ such that $x_0\in Int(C)$. On the other hand, since $D$ has no empty interior in $(C,\|.\|_{F})$, there exists $x_1\in D$ and $\epsilon_1>0$ such that $(B_F(x_1,\epsilon_1)\cap C) \subset D$. Using Proposition \ref{prop31}, we obtain that for all $\alpha\in (0,1]$, we have $\alpha x_0+(1-\alpha)x_1\in Int(C)$. Since $\alpha x_0+(1-\alpha)x_1\rightarrow x_1$ when $\alpha\rightarrow 0$, there exist some small $\alpha_0$ and an integer number $N\in \N_*$ such that $B_F(\alpha_0 x_0+(1-\alpha_0)x_1,\frac{1}{N})\subset (B_F(x_1,\epsilon_1)\cap C) \subset D$. Thus $D$ has a non empty interior in $F$.
\end{proof}
%%%%%%%%%%%%%%%%%%%%%%%%%%%%%%%%%%%%%%%%%%%%%%%%%%%%%%%%%%%%%%%
\subsection{A key lemma.} 
%%%%%%%%%%%%%%%%%%%%%%%%%%%%%%%%%%%%%%%%%%%%%%%%%%%%%%%%%%%%%
A map $p$ from a vector space $Z$ into $\R$ is said to be {\it subadditive} if and only if, for all $x, y\in Z$ one has
$$p(x+y)\leq p(x) + p(y).$$
A map $p$ is said to be {\it sublinear} if it is subadditive and satisfies $p(\lambda z)=\lambda p(z)$ for all $\lambda \geq 0$ and all $z\in Z$.
\vskip2mm
We give now our principal lemma. This lemma is based on the Baire category theorem. 
\begin{lemma}\label{lem33} Let $Z$ be a Banach space. Let $K$ be a non empty closed convex subset of $Z$ and suppose that $\textnormal{ri}(K)\neq \emptyset$. Let $\mathcal{T}$ be any nonempty set and $(p_n)_{n\in \mathcal{T}}$ be a collection of subadditive and lower semicontinuous functions on $Z$ and let $(\lambda_n)_{n\in \mathcal{T}}$ be a collection of nonnegative real number. Suppose that, for all $z\in K$, there exists $C_z \in \R$ such that, for all $n\in \mathcal{T}$, $p_n(z)\leq C_z \lambda_n$.\\
Then, for all $a \in K$, there exists $b_a \in \textnormal{Aff} (K)$ such that for all bounded subset $B$ of $\overline{\textnormal{Aff} (K)}$ there exists $R_{B}\geq 0$ such that  
$$\forall n \in \mathcal{T}, \hspace{4mm} \sup_{h\in B} p_n(h-a) \leq  R_{B} \cdot (\lambda_n + p_n(b_a -a)).$$
\end{lemma} 
%%%%%%%%%%%%%%%
\begin{proof} 
For each $m\in \N$, we set $F_m:=\left\{z\in K: \forall n \in \mathcal{T}, \;  p_n(z)\leq m \lambda_n \right\}$. 
The sets $F_m$ are closed subsets of $K$. Notice that 
$$F_m=K\cap \left(\bigcap_{n\in \mathcal{T}} p_n^{-1}(]-\infty, m \lambda_n])\right),$$
where, for each $n\in \mathcal{T}$, $p_n^{-1}(]-\infty, m \lambda_n])$ is a closed subset of $Z$ by the semicontinuity of $p_n$. On the other hand, we have $K=\bigcup_{m\in \N} F_m$. The inclusion $\supset$ is trivial. We prove the inclusion $\subset$. For each $z\in K$, there exists $C_z\in \R$ such that $p_n(z)\leq C_z \lambda_n$ for all $n\in \mathcal{T}$. If $C_z\leq 0$, we have that $z\in F_0$. If $C_z>0$, we put $m_1:= [C_z]+1$ where $[C_z]$ denotes the integer part of $C_z$, then we have that $z\in F_{m_1}$. We deduce then that for all $m\in \N$, the sets $F_m-a$ are closed subset of $K-a$ and that $K-a=\bigcup_{m\in \N} \left(F_m-a\right)$. Using the Baire category Theorem on the complete metric space $K-a$, we get an $m_0\in \N$ such that $F_{m_0}-a$ has a nonempty interior in $K-a$. Since by hypothesis $K-a$ has a nonempty interior in the normed vector subspace $F:=\textnormal{Aff} (K)-a$ of $Z$, then by using Proposition \ref{prop32} we obtain that $F_{m_0}-a$ has a nonempty interior in $F$. So there exists $z_0\in F_{m_0}-a$ and some integer number $N\in \N_*$ such that $B_{F}(z_0,\frac{1}{N}):= ( F\cap B_Z(z_0,\frac{1}{N})) \subset F_{m_0}-a$. In other words, for all $z \in B_F(b, \frac{1}{N})\subset F_{m_0}$ where $b:=a+z_0\in F_{m_0}\subset F$ and all $n\in \mathcal{T}$, we have:
\begin{equation} \label{eq31}
p_n(z)\leq m_0 \lambda_n.
\end{equation}
Now, let $B$ be a nonempty bounded subset of $F$, there exists an integer number $N_B\in \N_*$ such that $B\subset B_F(0,N_B)$. On the other hand, for all $h\in B$, there exists $z_h\in B_F(b,\frac{1}{N})$ such that $h=N_B N \cdot(z_h-b)$ (it sufficies to see that $z_h:= b+\frac{h}{N_B.N}\in B_F(b,\frac{1}{N})$). So using (\ref{eq31}) and the subadditivity of $p_n$, we obtain that, for all $n\in \mathcal{T}$:
\begin{eqnarray}  
p_n(h)&=& p_n(N_B N \cdot (z_h-b)) \nonumber\\
                    &\leq& N_BN \cdot p_n(z_h-b)\nonumber\\
                                    &\leq& N_BN \cdot (p_n(z_h)+p_n(-b))\nonumber\\
                                    &\leq& N_BN m_0 \lambda_n+ N_BN \cdot p_n(-b)\nonumber\\
                                    &\leq&  N_BN m_0 \lambda_n+ N_B N m_0 \cdot p_n(\frac{-b}{m_0})\nonumber\\
                                    &=&  N_BN m_0\left( \lambda_n+p_n(\frac{-b}{m_0})\right)\nonumber.
\end{eqnarray}
Setting $R_{B}:= N_B N m_0$ and $b_0:= \frac{-b}{m_0}\in F$ and by taking the supremum on $B$, we obtain for all $n \in \mathcal{T}$,
\begin{equation}\label{eq32}
\sup_{h\in B} p_n(h)\leq R_{B} \cdot ( \lambda_n+ p_n(b_0)).
\end{equation}
Now, let $\tilde{B}$ be any bounded subset of the closure $\overline{F}$ of $F$. There exists a bounded subset of $F$, $B$, such that $\tilde{B}=\overline{B}$. Hence for each $z\in \tilde{B}$, there exists a sequence $(h_k)_k$ in $B$ such that $h_k \rightarrow z$ when $k \rightarrow + \infty$. Thus, using the lower semicontinuity of $p_n$ for all $n\in \N$ and the inequality (\ref{eq32}), we obtain
$$p_n(z) \leq \liminf_{k\rightarrow +\infty} p_n(h_k) \leq \sup_{h\in B} p_n(h) \leq R_{B} \cdot( \lambda_n+ p_n(b_0)),$$
and by taking the supremum on $\tilde{B}$, we obtain, for all $n \in \mathcal{T}$,
$$\sup_{z\in \tilde{B}} p_n(z) \leq R_{B} \cdot ( \lambda_n+ p_n(b_0)).$$
Since $\overline{F}= \overline{\textnormal{Aff}(K)}-a$, by changing the bounded subsets $\tilde{B}$ of $\overline{F}$ by $B-a$, where $B$ is a bounded subset of $\overline{\textnormal{Aff}(K)}$ and by setting $b_a:= b_0 +a \in \textnormal{Aff}(K)$, we conclude the proof.
\end{proof}

We obtain the following corollary, which may be of interest in some cases.
\begin{corollary} \label{cor34} Let $Z$ be a Banach space and let $A$ be a non empty subset of $Z$. Let $\mathcal{T}$ be any nonempty set and $(p_n)_{n\in \mathcal{T}}$ be a collection of sublinear and lower semicontinuous functions on $Z$ and let $(\lambda_n)_{n \in \mathcal{T}}$ be a collection of nonegative real number. Let $C: Z\longrightarrow \R$ be a upper semicontinuous function. Suppose that
\begin{equation} \label{eq33}
 \forall n\in \mathcal{T},  \forall z\in A, \hspace{3mm} p_n(z)\leq C(z) \lambda_n. 
\end{equation} 
If $\textnormal{ri}(\overline{\textnormal{co}}(A)) \neq \emptyset$, then, for all $a\in K$, there exists $b_a\in \textnormal{Aff} (\overline{\textnormal{co}}(A))$ such that for all bounded subset $B$ of $\overline{\textnormal{Aff} (\overline{\textnormal{co}}(A))}$ there exists $R_{B}\geq 0$ such that  
$$\forall n\in \mathcal{T}, \hspace{4mm} \sup_{h\in B} p_n(h-a) \leq R_{B} \cdot (\lambda_n + p_n(b_a -a)).$$
\end{corollary}
\begin{proof} We can apply Lemma \ref{lem33}, with $K=\overline{\textnormal{co}}(A)$. For this, it suffices to establish that 
$$\forall n \in \mathcal{T}, \forall z \in \overline{\textnormal{co}}(A), \hspace{3mm} p_n(z)\leq C(z) \lambda_n.$$
The previous inequality is obtained by using (\ref{eq33}), the sublinearity and semicontinuity of $p_n$ for all $n\in \N$, together with the upper semicontinuity of the function $C$.
\end{proof}
%%%%%%%%%%%%%%%%%%%%%%%%%%%%%%%%%%%%%%%%%%%%%%%%%%%%%%%%%%%%%
%%%%%%%%%%%%%%%%%%%%%%%%%%%%%%%%%%%%%%%%%%%%%%%%%%%%%%%%%%%%%
%%%%%%%%%%%%%%%%%%%%%%%%%%%%%%%%%%%%%%%%%%%%%%%%%%%%%%%%%%%%%
\subsection{Preliminaries for multipliers in infinite horizon} 
%%%%%%%%%%%%%%%%%%%%%%%%%%%%%%%%%%%%%%%%%%%%%%%%%%%%%%%%%%%%%%
%%%%%%%%%%%%%%%%%%%%%%%%%%%%%%%%%%%%%%%%%%%%%%%%%%%%%%%%%%%%%%%
%%%%%%%%%%%%%%%%%%%%%%%%%%%%%%%%%%%%%%%%%%%%%%%%%%%%%%%%%%%%%%%
As consequence of Lemma \ref{lem33}, we obtain the following proposition. The sequences $(\lambda_n)_n \in (\R^+)^{\N}$ and $(f_n)_n \in (Z^*)^{\N}$ in the following result, correspond to the multipliers.
\begin{proposition}\label{prop35} Let $Z$ be a Banach space. Let $(f_n)_n \in (Z^*)^{\N}$ be a sequence of linear continuous functionnals on $Z$ and let $(\lambda_n)_n \in (\R^+)^{\N}$ such that $\lambda_n \rightarrow 0$ when $n \rightarrow + \infty$. Let $K$ be a non empty closed convex subset of $Z$ such that $ri( K)\neq \emptyset$. Suppose that 
\begin{itemize}   
\item[$(1)$] for all $z\in K$, there exists a real number $C_z$ such that, for all $n\in \N$, we have $f_n(z)\leq C_z \lambda_n$.
\item[$(2)$] $f_n\stackrel{w^*}{\rightarrow}0$ when $n \rightarrow + \infty$.
\end{itemize}
Let $a \in K$ and set $X:=\overline{\textnormal{Aff} (K)}-a$. Then, we have,
\begin{itemize} 
\item[(i)] $\|(f_n)_{|X}\|_{X^*}\rightarrow 0$ when $n \rightarrow + \infty$.
\item[(ii)] If moreover we assume that the codimension of $X$ in $Z$ is finite, then $\|f_n\|_{Z^*}\rightarrow 0$  when $n \rightarrow + \infty$.
\end{itemize}
\end{proposition}
\begin{proof} Using Lemma \ref{lem33} with $\mathcal{T}=\N$, the linear continuous functions $f_n$ and the bounded set $B:=S_{X} +a$ of $\overline{\textnormal{Aff} (K)}$ (where, $S_{X}$ denotes the sphere of $X$), we get a point $b_0$ depending only on $X$ and a constant $R_B\geq 0$ such that 
$$\|(f_n)_{|X}\|_{X^*}=\sup_{\|h\|_X=1} f_n(h) \leq R_B \cdot (\lambda_n + f_n(b_0)).$$
Since $f_n\stackrel{w^*}{\rightarrow}0 \; (n \rightarrow + \infty)$ and $\lambda_n \rightarrow 0 \; (n \rightarrow + \infty)$ we obtain that $\|(f_n)_{|X}\|_{X^*}\rightarrow 0 \; (n \rightarrow + \infty)$. Suppose now that $X$ is of finite codimension in $Z$, then there exists a finite-dimensional subspace $E$ of $Z$, such that $Z= X\oplus E$. Thus, there exists $L> 0$ such that
$$\|f_n\|_{Z^*}\leq L\left(\|(f_n)_{|E}\|_{E^*}+\|(f_n)_{|X}\|_{X^*}\right).$$
Since $f_n\stackrel{w^*}{\rightarrow}0 \; (n \rightarrow + \infty)$ and since the weak-star topology and the norm topology coincids on $E$ since its dimension is finite, we have that $\|(f_n)_{|E}\|_{E^*}\longrightarrow 0 \; (n \rightarrow + \infty)$. On the other hand, we proved above that $\|(f_n)_{|X}\|_{X^*}\rightarrow 0$. Thus, $\|f_n\|_{Z^*}\longrightarrow 0 \; (n \rightarrow + \infty)$. 
\end{proof}
\begin{remark}\label{rem36}
Proposition \ref{prop35} shows that under the condition $(1)$, we have that $f_n\stackrel{w^*}{\not\rightarrow} 0$, whenever $\|(f_n)_{|X}\|_{X^*}\not\rightarrow 0.$ If moreover, $X$ is of finite codimension in $Z$, then $f_n \stackrel{w^*}{\not\rightarrow}0$, whenever $\|f_n\|_{Z^*}\not\rightarrow 0.$ Thus, the condition $(1)$ is a criterion ensuring that a sequence of norm one in an infinite dual Banach space, does not converges to $0$ in the weak$^*$ topology.
\end{remark}
To ensure that the multipliers are nontrivial at the limit, the authors in \cite{MM} used a lemma from [\cite{LY}, pp. 142, 135], which can be recovered by taking $C(z)=1$ for all $z\in Z$ in the following corollary.
\begin{definition}\label{def37} A subset $Q$ of a Banach space $Z$ is said to be of finite codimension in $Z$ if there exists a point $z_0$ in the closed convex hull of $Q$ such that the closed vector space generated by $Q- z_0:= \left\{q - z_0 | \hspace{1mm} q \in Q\right\}$ is of finite codimension in $Z$ and the closed convex hull of $Q-z_0$ has a no empty interior in this vector space.
\end{definition}
\begin{corollary}\label{cor38}  Let $Q\subset Z$ be a subset of finite codimension in $Z$. Let $C: Z\longrightarrow \R$ be a upper semicontinuous function. Let $\delta > 0$, $(f_k)_k \in (Z^*)$ and $\lambda_k \geq 0$, $\lambda_k \rightarrow 0 \; (k \rightarrow + \infty)$ such that
\begin{itemize}
\item[$(i)$] $\|f_k\| \geq \delta$, for all $k\in \N$ and $f_k\stackrel{w^*}{\rightarrow} f \; (k \rightarrow + \infty)$.
\item[$(ii)$] For all $z\in Q$, and for all $k\in \N$, $f_k(z)\leq C(z)\lambda_k $.
\end{itemize}
Then, $f \neq 0$.
\end{corollary}
\begin{proof} First, note that from the condition $(ii)$, the linearity and continuity of $f_k$, $k\in \N$ and the upper semicontinuity of $C$, we have also that, for all $z\in \overline{\textnormal{co}}(Q)$ and for all $k \in \N$, $f_k(z)\leq C(z)\lambda_k $. Suppose by contradiction that $f=0$, then using Proposition \ref{prop35} and the fact that $Q$ is of finite codimension in $Z$, we get that $\|f_k\|_{Z^*}\rightarrow 0 \; (n \rightarrow + \infty)$, which contredicts the condition $(i)$.
\end{proof}
The following proposition is used in the proof of our main result Theorem \ref{thm22}. In Proposition \ref{prop39}, the sequence $(\beta^{(n)})_{n\geq 2}$ in$ (\R^+)$ and the list $(f_t^{(n)})_{1 \leq t \leq n+1} \in (X^*)^{n+1}$, correspond to the non trivial  multipliers at the finite horizon $n$, for all $n \geq 2$. The aim is to find conditions under which, these sequences have subsequences which converge to  non trivial multipliers at the infinite horizon.
\begin{proposition} \label{prop39} Let $Z$ be a separable Banach space and $Z^*$ its topological dual. Let $K$ be a closed convex subset of $Z$ such that $ri( K)\neq \emptyset$ and that $\overline{\textnormal{Aff} (K)}$ is of finite codimension in $Z$. 
Let $(\beta^{(n)})_{n\geq 2}$ be a sequence of nonegative real number and $(f_t^{(n)})_{1 \leq t \leq n+1} \in (Z^*)^{n+1}$, for all $n \geq 2$.  Let $s\in \N_*$ be a fixed natural number. Suppose that:
\begin{itemize}
\item[$(1)$] for all $n\geq 2$, $\beta^n+\|f_s^n\|_{Z^*}=1$,
\item[$(2)$] there exists $ a_t, b_t\geq 0 $ such that $ \|f_{t}^n\|_Z\leq a_t \beta^n+b_t\|f_s^n\|_Z$ for all $n\geq 2$ and for all $1\leq t \leq n+1$,
\item[$(3)$] for all $z\in K$, there exist a real number $c_z$ such that: $f_s^n(z)\leq c_z \beta^n$ for all $n\geq 2$.
\end{itemize}
Then there exist a strictly increasing map $k\mapsto n_k$, from $\N$ into $\N$, $\beta \in \R^+$ and $(f_t)_{t\geq 1} \in (Z^*)^{\N}$ such that: 
\begin{itemize}
\item[$(i)$] $\beta^{n_k} \longrightarrow \beta$ when $k\rightarrow +\infty$,
\item[$(ii)$] for each $t \in \N$, $f^{n_k}_t \stackrel{w^*}{\longrightarrow} f_t$ when $k\rightarrow +\infty$,
\item[$(iii)$] $(\beta,f_s) \neq (0,0)$.
\end{itemize}
\end{proposition}
\begin{proof} From $(1)$ and $(2)$ we get that, for each $t\geq 1$, the sequences $(f^n_t)_{1\leq t \leq n+1}$ and $(\lambda^n_0)_{n \geq 2}$ are bounded. Hence, using the Banach-Alaoglu theorem and the diagonal process of Cantor, we get a strictly increasing map $k\mapsto n_k$, from $\N$ into $\N$, a nonegative real number $\beta\in \R^+$, and a sequence $(f_t)_{t\geq 1}\in (Z^*)^{\N_*}$ satisfying $(i)$ and $(ii)$. Suppose by contradiction that  $(\beta,f_s)=(0,0)$, i.e. $\beta^{n_k}\longrightarrow 0$ and $f^{n_k}_s \stackrel{w^*}{\longrightarrow} 0$ when $k\rightarrow +\infty$. Using the condition $(3)$ and Proposition \ref{prop35} we have that $\|f^{n_k}_s\|_{Z^*}\longrightarrow 0$ when $k\rightarrow +\infty$. Since $\beta^{n_k}\longrightarrow 0$ when $k\rightarrow +\infty$, we have also $\beta^{n_k}+\|f^{n_k}_s\|_{Z^*}\longrightarrow 0$ which is a contradiction with the condition $(1)$.
\end{proof}
\section{Multipliers}
In this section, after the recall of the method of reduction to finite horizon, we establish multiplier rules (Lemma \ref{lem45} and Lemma \ref{lem46}), in the spirit of Fritz John's theorem, for the problems of finite horizon.
\vskip2mm

First we recall the method of reduction to finite horizon. When $((\hat{x}_t)_{t \in \N}, (\hat{u}_t)_{t \in \N})$ is an optimal solution of (${\bf PE}_k(\sigma)$), $k \in \{1,2,3 \}$, we build the following finite-horizon problem.
\[
({\bf EF}(\sigma))
\left\{
\begin{array}{rl}
{\rm Maximize} & J^T(x_1,...,x_T, u_0,u...,u_T) := \sum_{t=0}^T \phi_t(x_t,u_t)\\
{\rm when} & \forall t \in \{0,...,T \}, \; x_{t+1} = f_t(x_t,u_t)\\
\null & x_0 = \sigma, \; x_{T+1} = \hat{x}_{T+1}.
\end{array}
\right.
\]
Similarly, when $((\hat{x}_t)_{t \in \N}, (\hat{u}_t)_{t \in \N})$ is an optimal solution of (${\bf PI}_k(\sigma)$), $k \in \{1,2,3 \}$, we build the following finite-horizon problem
\[
({\bf IF}(\sigma)) 
\left\{
\begin{array}{rl}
{\rm Maximize} & J^T(x_1,...,x_T, u_0,u...,u_T) := \sum_{t=0}^T \phi_t(x_t,u_t)\\
{\rm when} & \forall t \in \{0,...,T \}, \; x_{t+1} \leq f_t(x_t,u_t)\\
\null & x_0 = \sigma, \; x_{T+1} = \hat{x}_{T+1}.
\end{array}
\right.
\]
The proof of the following result is similar to the proof given in \cite{JBHC}.
\begin{lemma}\label{lem41}
Let $k \in \{1,2,3 \}$. When $((\hat{x}_t)_{t \in \N}, (\hat{u}_t)_{t \in \N})$ is an optimal solution of (${\bf PE}_k(\sigma)$) (respectively (${\bf PI}_k(\sigma)$)), for all $T \in \N$, $T \geq 2$, then the restriction \\
$(\hat{x}_1,..., \hat{x}_T, \hat{u}_0,..., \hat{u}_T)$ is an optimal solution of (${\bf EF}(\sigma$)  (respectively (${\bf IF}(\sigma$))). 
\end{lemma}
To work on these problems, we introduce several notations. We write ${\bf x}^T :=(x_1,...,x_T) \in \prod_{t=1}^T X_t$ and ${\bf u}^T := (u_0,...,u_T) \in \prod_{t=0}^T U_t$. For all $t \in \{0,...,T \}$, we define the mapping $g^T_t : \prod_{t=1}^T X_t \times \prod_{t=0}^T U_t \rightarrow X_{t+1}$ by setting 
\begin{equation}\label{eq41}
g^T_t({\bf x}^T, {\bf u}^T) := 
\left\{
\begin{array}{lcl}
-x_1 + f_0(\sigma, u_0) & {\rm if} & t=0\\
-x_{t+1} + f_t(x_t,u_t) & {\rm if} & t \in \{1,...,T-1 \}\\
-\hat{x}_{T+1} + f_T(x_T,u_T) & {\rm if} & t=T.
\end{array}
\right.
\end{equation}
We define $g^T : \prod_{t=1}^T X_t \times \prod_{t=0}^T U_t \rightarrow \prod_{t=0}^T X_t$ by setting
\begin{equation}\label{eq42}
g^t({\bf x}^T, {\bf u}^T) := (g^T_0({\bf x}^T, {\bf u}^T), ..., g^T_T({\bf x}^T, {\bf u}^T)).
\end{equation}
And so the problem (${\bf EF}(\sigma)$) is exactly
\begin{equation}\label{eq43}
\left\{
\begin{array}{cl}
{\rm Maximize}& J^T({\bf x}^T, {\bf u}^T)\\
{\rm when} & g^T({\bf x}^T, {\bf u}^T) = 0
\end{array}
\right.
\end{equation}
and the problem problem $({\bf IF}(\sigma))$ is exactly
\begin{equation}\label{eq44}
 \left\{
\begin{array}{cl}
{\rm Maximize}& J^T({\bf x}^T, {\bf u}^T)\\
{\rm when} & g^T({\bf x}^T, {\bf u}^T) \geq  0
\end{array}
\right.
\end{equation}
Under (A3), $g^T$ is of class $C^1$ at $({\bf \hat{x}}^T, {\bf \hat{u}}^T)$ as a composition of mappings of class $C^1$, and the calculation of its differential gives
$$Dg^T({\bf x}^T, {\bf u}^T) \cdot ({\bf \delta x}^T, {\bf \delta u}^T)= (Dg^T_0({\bf x}^T, {\bf u}^T) \cdot ({\bf \delta x}^T, {\bf \delta u}^T)       ,...,Dg^T_T({\bf x}^T, {\bf u}^T) \cdot ({\bf \delta x}^T, {\bf \delta u}^T))$$
and we have 
$$Dg^T_0({\bf x}^T, {\bf u}^T) \cdot ({\bf \delta x}^T, {\bf \delta u}^T) = - \delta x_1 + D_2f_0(\sigma, u_0) \cdot \delta u_0,$$
and when $t \in \{1,...,T-1 \}$, 
$$Dg^T_t({\bf x}^T, {\bf u}^T) \cdot ({\bf \delta x}^T, {\bf \delta u}^T) = - \delta x_{t+1} + D_1f_t(x_t,u_t) \cdot \delta x_t + D_2 f_t(x_t,u_t) \cdot\delta u_t,$$
and 
$$Dg^T_T({\bf x}^T, {\bf u}^T) \cdot ({\bf \delta x}^T, {\bf \delta u}^T) = D_1f_T(x_T,u_Y) \cdot \delta x_T + D_2 f_T(x_T,u_T) \cdot \delta u_T.$$
Thus in order to study Im$Dg^T({\bf \hat{x}}^T, {\bf \hat{u}}^T)$ we need to treat the equation 
$$Dg^T({\bf \hat{x}}^T, {\bf \hat{u}}^T) \cdot ({\bf \delta x}^T, {\bf \delta u}^T) = (b_1,...,b_{T+1}).$$
It is the following system
\begin{equation}\label{eq45}
\left\{
\begin{array}{rcl}
b_1 &=& - \delta x_1 + D_2f_0(\sigma, u_0) \cdot \delta u_0 \\
 b_2 &=& - \delta x_2 + Df_1(\hat{x}_1, \hat{u}_1) \cdot (\delta x_1,\delta u_1) \\
.... \\
b_T &=& - \delta x_T + Df_{T-1}(\hat{x}_{T-1}, \hat{u}_{T-1}) \cdot  (\delta x_{T-1}, \delta u_{T-1})\\
b_{T+1} &=& Df_{T}(\hat{x}_{T}, \hat{u}_{T}) \cdot (\delta x_{T}, \delta u_{T}).
\end{array}
\right.
\end{equation}
\begin{lemma}\label{lem42} Under (A1) and (A3), the set Im$D_1g^T({\bf \hat{x}}^T, {\bf \hat{u}}^T)$ is closed into $X^{T+1}$.
\end{lemma}
\begin{proof}
Suppose that a sequence $((b^n_1,...,b^n_{T+1}))_n \in ({\rm Im}D_1g^T({\bf \hat{x}}^T, {\bf \hat{u}}^T))^{\N}$ converges to some $(b_1,b_2,...,b_{T+1})$. We prove that $(b_1,b_2,...,b_{T+1})\in {\rm Im}D_1g^T({\bf \hat{x}}^T, {\bf \hat{u}}^T)$. Indeed, there exists $(\delta x^n_1,\delta x^n_2,..., \delta x^n_{T}) \in X^T$ satisfying
\begin{equation}\label{eq46}
\left.
\begin{array}{rcl}
b^n_1 &=& - \delta x^n_1 \\
 b^n_2 &=& - \delta x^n_2 + Df_1(\hat{x}_1, \hat{u}_1) \cdot \delta x^n_1 \\
.... \\
b^n_T &=& - \delta x^n_T + Df_{T-1}(\hat{x}_{T-1}, \hat{u}_{T-1}) \cdot \delta x^n_{T-1}\\
b^n_{T+1} &=& Df_{T}(\hat{x}_{T}, \hat{u}_{T}) \cdot \delta x^n_{T}.
\end{array}
\right\}
\end{equation}
Since $(b^n_1)_n$ converges to $b_1$, we get that $(\delta x^n_1)_n$ converges to some $\delta x_1$ and so $(Df_1(\hat{x}_1, \hat{u}_1) \cdot \delta x^n_1)_n$ converges to $Df_1(\hat{x}_1, \hat{u}_1) \cdot \delta x_1$ by continuity. Since $(b^n_2)_n$ converges to $b_2$, we get that $(\delta x^n_2)_n$ converges to some $\delta x_2$ and so $b_2=- \delta x_2 + Df_1(\hat{x}_1, \hat{u}_1) \cdot \delta x_1$. We proceed inductively to obtain 
\begin{equation}\label{eq47}
\left.
\begin{array}{rcl}
b1 &=& - \delta x_1 \\
 b_2 &=& - \delta x_2 + Df_1(\hat{x}_1, \hat{u}_1) \cdot \delta x_1 \\
.... \\
b_T &=& - \delta x_T + Df_{T-1}(\hat{x}_{T-1}, \hat{u}_{T-1}) \cdot \delta x_{T-1}\\
b_{T+1} &=& Df_{T}(\hat{x}_{T}, \hat{u}_{T}) \cdot \delta x_{T}.
\end{array}
\right\}
\end{equation}
This shows that $(b_1,b_2,...,b_{T+1})\in {\rm Im}D_1g^T({\bf \hat{x}}^T, {\bf \hat{u}}^T)$ and conclude the proof.
\end{proof}
The proof of the following result is similar to the proof of Lemma 3.10 in \cite{BB}, replacing Lemma 3.5 in \cite{BB} by Lemma \ref{lem42}.
\begin{lemma}\label{lem43} Under (A1), (A3) and (A5), the range Im$Dg^T({\bf \hat{x}}^T, {\bf \hat{u}}^T)$ is closed in $X^{T+1}$.
\end{lemma}
The following theorem was established in the book of Jahn \cite{Ja} (Theorem 5.3 in p.106-111, and Theorem 5.6, p. 118).
\begin{theorem}\label{thm44}
Let $\Xi$, $Y$ and $Z$ three real Banach spaces, and $\hat{\xi} \in \Xi$. We assume that the following conditions are fulfilled.
\begin{enumerate}
\item $Y$ is ordered by a cone $C$ with a nonempty interior.
\item $\hat{S}$ is a convex subset of $\Xi$ with a nonempty interior.
\item ${\mathcal I} : \Xi \rightarrow \R$ is a functional which is Fr\'echet differentiable at $\hat{\xi}$.
\item $\Gamma : \Xi \rightarrow Y$ is a mapping which is Fr\'echet differentiable at $\hat{\xi}$.
\item $H :  \Xi \rightarrow Z$ is a mapping which is Fr\'echet differentiable at $\hat{\xi}$.
\item $S := \{ \xi \in \hat{S} : \Gamma (\xi) \in - C, H(\xi) = 0 \}$ is nonempty.
\item Im$DH(\hat{\xi})$ is closed into $Z$.
\end{enumerate}
If $\hat{\xi}$ is a solution of the following minimization problem 
\[
\left\{
\begin{array}{cl}
{\rm Minimize} & {\mathcal I}(\xi)\\
{\rm when} & \xi \in S
\end{array}
\right.
\]
then there exist $\lambda_0 \in [0, + \infty)$, $\Lambda_1 \in Y^*$ a positive linear functional, $\Lambda_2 \in Z^*$ such that the following conditions are satisfied:
\begin{enumerate}
\item[(i)] $(\lambda_0, \Lambda_1, \Lambda_2) \neq (0,0,0)$
\item[(ii)] $\langle  \lambda_0 D{\mathcal I}( \hat{\xi}) + \Lambda_1 \circ D \Gamma ( \hat{\xi}) + \Lambda_2 \circ DH( \hat{\xi}), \xi - \hat{\xi} \rangle \leq 0$ for all $\xi \in S$.
\end{enumerate}
\end{theorem}  
\vskip1mm
\begin{lemma}\label{lem45}
Let $((\hat{x}_t)_{t \in \N}, (\hat{u}_t)_{t \in \N})$ be an optimal process of $({\bf PE}_k)(\sigma))$ when $k \in \{1,2,3 \}$. Under (A1), (A3) and (A5), we assume moreover that $U_t$ is convex for all $t \in \N$. Then, for all $T \in \N$, $T \geq 2$, there exist $\lambda^{T}_0 \in \R$ and 
$(p^{T}_t)_{1 \leq t \leq T+1} \in (X^*)^{T+1}$ such that the following conditions hold.
\begin{enumerate}
\item[(a)] $\lambda^{T}_0$ and $(p^{T}_t)_{1 \leq t \leq T+1}$ are not simultaneously equal to zero.
\item[(b)] $\lambda^{T}_0 \geq 0$.
\item[(c)] $p^{T}_t = p^{T}_{t+1} \circ D_1f_t(\hat{x}_t, \hat{u}_t) + \lambda^{T}_0.D_1 \phi_t(\hat{x}_t, \hat{u}_t)$  for all $t \in \{1,...,T \}$.
\item[(d)] $\langle \lambda^{T}_0.D_2\phi_t(\hat{x}_t, \hat{u}_t) + p^{T}_{t+1} \circ D_2f_t(\hat{x}_t, \hat{u}_t), u_t - \hat{u}_t \rangle \leq 0$ for all $t \in \{0,...,T \}$ and for all $u_t \in U_t$.
\end{enumerate}
\end{lemma}
\begin{proof} 
Using Lemma \ref{lem41}, we know that $({\bf \hat{x}}^T, {\bf \hat{u}}^T) = (\hat{x}_1,..., \hat{x}_T, \hat{u}_0,..., \hat{u}_T)$ is an optimal solution of (${\bf EF}(\sigma$)). We want to use Theorem \ref{thm44} where the inequality constraints are absent, and so we don't nee to the first assumption of Theorem \ref{thm44}, and among the conclusions we lost that the $p_t$ are positive. We have not inequality constraints and so we can delete $\Gamma$ and conditions on the cone $C$, and we have $H = g^T$. Using Lemma \ref{lem43}, we know that Im$Dg^T({\bf \hat{x}}^T, {\bf \hat{u}}^T)$ is closed in $X^{T+1}$. And so there exists $ \lambda_0 \in [0, + \infty)$ (that is the conclusion (ii)) and $\Lambda_2 \in (X^*)^{T+1}$ such $(\lambda_0 , \Lambda_2) \neq (0,0)$. Denoting by $p_t^{(T)}$ the coordinates of $\Lambda_2$ in $X^*$, we obtain the conclusion (i). From conclusion (ii) of Theorem \ref{thm44}, using the partial differentials with respect to ${\bf u}^T$ and with respect to ${\bf u}^T$, and using the openess of $\prod_{t=1}^T X_t$, we obatin
$$\lambda_0 D_1J^T({\bf \hat{x}}^T, {\bf \hat{u}}^T) +  \Lambda_2 \circ Dg^T_1({\bf \hat{x}}^T, {\bf \hat{u}}^T) = 0$$
$$\langle \lambda_0 D_2J^T({\bf \hat{x}}^T, {\bf \hat{u}}^T) +  \Lambda_2 \circ Dg^T_2({\bf \hat{x}}^T, {\bf \hat{u}}^T), {\bf u}^T - {\bf \hat{u}}^T \rangle \leq 0$$
for all ${\bf u}^T \in \prod_{t=0}^T U_t$. This gives the conclusions $(c)$ and $(d)$.
\end{proof}
\begin{lemma}\label{lem46}
Under (A1), (A3) and (A5), we assume moreover that $U_t$ is convex for all $t \in \N$ and that $Int(X_+) \neq \emptyset$. Then, for all $T \in \N$, $T \geq 2$, there exist $\lambda^{T}_0 \in \R$ and 
$(p^{T}_t)_{1 \leq t \leq T+1} \in (X^*)^{T+1}$ such that the following conditions hold.
\begin{enumerate}
\item[(a)] $\lambda^{T}_0$ and $(p^{T}_t)_{1 \leq t \leq T+1}$ are not simultaneously equal to zero.
\item[(b)] $\lambda^{T}_0 \geq 0$, and $p_t \geq 0$ for all $t \in \{1, ...,T+1 \}$.
\item[(c)] $p^{T}_t = p^{T}_{t+1} \circ D_1f_t(\hat{x}_t, \hat{u}_t) + \lambda^{T}_0 D_1 \phi_t(\hat{x}_t, \hat{u}_t)$  for all $t \in \{1,...,T \}$.
\item[(d)] $\langle \lambda^{T}_0 D_2\phi_t(\hat{x}_t, \hat{u}_t) + p^{T}_{t+1} \circ D_2f_t(\hat{x}_t, \hat{u}_t), u_t - \hat{u}_t \rangle \leq 0$ for all $t \in \{0,...,T \}$ and for all $u_t \in U_t$.
\end{enumerate}
\end{lemma}
\begin{proof} We procced as in the proof of Lemma \ref{lem45} without deleting the inequality constraints, but deleting the equality constraints.
\end{proof}
We need the following lemma for the proof of our main result Theorem \ref{thm22}.
\begin{lemma}\label{lem47} Under the assumptions of Lemma \ref{lem45} or Lemma \ref{lem46}, suppose moreover that (A4) is satisfied.  Then, for all $T \in \N$, $T \geq 2$, there exist $\lambda^{T}_0 \in \R$ and 
$(p^{T}_t)_{1 \leq t \leq T+1} \in (X^*)^{T+1}$ such that the following conditions hold.
\begin{itemize}
\item[(1)] For all $T\geq 2$, for all $s \in \{1,...,T \}$ and all $1 \leq t \leq T+1$, there exists $a_t, b_t\geq 0$ such that $ \|p_{t}^T\| \leq a_t \lambda_0^T+b_t\|p_s^T\|$.
\item[(2)] For all $s \in \{1,...,T \}$, $(\lambda_0^T,p^T_s) \neq 0$.
\item[(3)] For all $s \in \{1,...,T \}$, for all $z\in A_s:=D_2 f_{s-1}(\hat{x}_{s-1}, \hat{u}_{s-1})(T_{U_{s-1}}(\hat{u}_{s-1}))$, there exists $C_z\in \R$ such that: $ \forall T\geq 2, \hspace{3mm} p^T_{s}(z) \leq C_z \lambda_0^T$.
\end{itemize}
\end{lemma}
\begin{proof} By adding $(c)$ and $(d)$ of Lemma \ref{lem45} (respectively Lemma \ref{lem46}) we obtain for all $t \in \{1,...,T \}$, for all $h\in X$ and for all $u_t \in U_t$ 
\[
\begin{array}{cl}
\null &\langle p^{T}_{t+1}, D_1f_t(\hat{x}_t, \hat{u}_t)(h) +  D_2f_t(\hat{x}_t, \hat{u}_t) \cdot (u_t - \hat{u}_t)\rangle \\
 \null & +\lambda^{T}_0.[D_1 \phi_t(\hat{x}_t, \hat{u}_t)(h) + D_2\phi_t(\hat{x}_t, \hat{u}_t) \cdot ( u_t - \hat{u}_t )]\\
\leq & p^{T}_t(h).
\end{array}
\]
Equivalently, for all $t \in \{1,...,T \}$ and for all $(h, k)\in X\times T_{U_t}(\hat{u}_t)$
\begin{equation} \label{eq48}
\langle p^{T}_{t+1}, Df_t(\hat{x}_t, \hat{u}_t) \cdot (h,k)\rangle \leq p^{T}_t(h)- \lambda^{T}_0 D \phi_t(\hat{x}_t, \hat{u}_t)(h,k).
\end{equation}
Thus we get for all $t \in \{1,...,T \}$ and for all $(h, k)\in X\times T_{U_t}(\hat{u}_t)$
\begin{equation} \label{eq49}
\langle p^{T}_{t+1}, Df_t(\hat{x}_t, \hat{u}_t) \cdot (h,k)\rangle \leq \|p^{T}_t\|\|h\|_X + \lambda^{T}_0 \|D \phi_t(\hat{x}_t, \hat{u}_t)\| \cdot \|(h,k)\|_{X\times U}.
\end{equation}
Since, for all $t \in \N_*$, $0\in Int\left(Df_t(\hat{x}_t, \hat{u}_t)((X\times T_{U_t}(\hat{u}_t))\cap B_{X\times U})\right)$, there exists a constant $r_t >0$ such that $B_{X}(0,r_t)\subset Df_t(\hat{x}_t, \hat{u}_t)((X\times T_{U_t}(\hat{u}_t))\cap B_{X\times U})$. Thus, from $(\ref{eq49})$ we obtain 
\begin{equation} \label{eq410}
\|p^{T}_{t+1}\|\leq \frac{1}{r_t}(\|p^{T}_t\|+ \lambda^{T}_0\|D \phi_t(\hat{x}_t, \hat{u}_t)\|).
\end{equation}
On the other hand, using $(c)$ of Lemma \ref{lem45} (respectively Lemma \ref{lem46}), we get, for all $t \in \lbrace 1,...,T \rbrace$, 
\begin{equation} \label{eq411}
\|p^{T}_t \| \leq \|p^{T}_{t+1}\| \cdot  \|D_1f_t(\hat{x}_t, \hat{u}_t)\| + \lambda^{T}_0 \|D_1 \phi_t(\hat{x}_t, \hat{u}_t)|\|.  
\end{equation}
Thus, by combining $(\ref{eq410})$ and  $(\ref{eq411})$ for all $T\geq 2$, for all $s \in \{1,...,T \}$, and all $1 \leq t \leq T+1$, there exist $a_t, b_t\geq 0$ such that 
$$ \|p_{t}^T\|\leq a_t \lambda_0^T+b_t\|p_s^T\|.$$
This gives the part $(1)$. Suppose that there exists $s \in \{1,...,T \}$ such that $(\lambda_0^T,p^T_s)=(0,0)$. Using the above inequality we obtain that $\lambda^{T}_0$ and $(p^{T}_t)_{1 \leq t \leq T+1}$ are simultaneously equal to zero which contredicts the part $(a)$ of Lemma \ref{lem45} (respectively Lemma \ref{lem46}). Thus, $(\lambda_0^T,p^T_s)\neq (0,0)$ which gives the part $(2)$. 
\vskip1mm
Now, using $(d)$ of Lemma \ref{lem45} (respectively Lemma \ref{lem46}) for an arbitrary $s \in \{1,...,T \}$, for all $T\geq 2$, and for all $u_s \in U_s$, we have
$$\langle p^{T}_{s} \circ D_2f_{s-1}(\hat{x}_{s-1}, \hat{u}_{s-1}), u_{s-1} - \hat{u}_{s-1} \rangle \leq -\langle \lambda^{T}_0 D_2\phi_{s-1}(\hat{x}_{s-1}, \hat{u}_{s-1}), u_{s-1} - \hat{u}_{s-1} \rangle.$$
For all $z\in A_s:=D_2 f_{s-1}(\hat{x}_{s-1}, \hat{u}_{s-1})(T_{U_{s-1}}(\hat{u}_{s-1}))$, using the definition of the set $T_{U_{s-1}}(\hat{u}_{s-1})$, there exist $(u^{y_k}_{s-1})_k \in U_{s-1}^{\N}$ and $(\alpha_k)_k \in (\R^+)^{\N}$ such that $y_z:=\lim_{k \rightarrow + \infty} (\alpha_k (u^{y_k}_{s-1}-\hat{u}_{s-1}))$ and $z=D_2 f_{s-1}(\hat{x}_{s-1}, \hat{u}_{s-1}) \cdot y_z$. So, using the above inequality and doing $k \rightarrow + \infty$, we get 
$$p^{T}_{s}(z) \leq -\langle \lambda^{T}_0 D_2\phi_{s-1}(\hat{x}_{s-1}, \hat{u}_{s-1}), y_z \rangle,$$ 
and so there exists 
$$C_z:= -\langle D_2\phi_{s-1}(\hat{x}_{s-1}, \hat{u}_{s-1}), y_z \rangle$$
such that, for all $T\geq 2$, we have $p^T_{s}(z) \leq C_z \lambda_0^T$. This gives the part $(3)$.
\end{proof}
%%%%%%%%%%%%%%%%%%%%%%%%%%%%%%%%%%%%%%%%%%%%%%%%%%%%
\section{The proof of the main results.}
%%%%%%%%%%%%%%%%%%%%%%%%%%%%%%%%%%%%%%%%%%%%%%%%%%%
%%%%%%%%%%%%%%%%%%%%%%%%%%%%%%%%%%%%%%%%%%%%%%%%%%%%
%%%%%%%%%%%%%%%%%%%%%%%%%%%%%%%%%%%%%%%%%%%%%%%%%
This section is devoted to the proofs of the Pontryagin principle for systems governed by a difference equation, and of the Pontryagin principle for systems governed by a difference inequation%
\begin{proof}[\bf Proof of Theorem \ref{thm22}] Let us prove the existence of the sequences $(p_t)_{t \in \N_*} \in (X^*)^{\N_*}$ and $\lambda_0 \geq0$ satisfying the theorem. By Lemma \ref{lem45}, for all $T \in \N$, $T \geq 2$, there exist $\lambda^{T}_0 \in \R$ and $(p^{T}_t)_{1 \leq t \leq T+1} \in (X^*)^{T+1}$ such that the following conditions hold.
\begin{enumerate}
\item[(a)] $\lambda^{T}_0$ and $(p^{T}_t)_{1 \leq t \leq T+1}$ are not simultaneously equal to zero.
\item[(b)] $\lambda^{T}_0 \geq 0$.
\item[(c)] $p^{T}_t = p^{T}_{t+1} \circ D_1f_t(\hat{x}_t, \hat{u}_t) + \lambda^{T}_0 D_1 \phi_t(\hat{x}_t, \hat{u}_t)$  for all $t \in \{1,...,T \}$.
\item[(d)] $\langle \lambda^{T}_0.D_2\phi_t(\hat{x}_t, \hat{u}_t) + p^{T}_{t+1} \circ D_2f_t(\hat{x}_t, \hat{u}_t), u_t - \hat{u}_t \rangle \leq 0$ for all $t \in \{0,...,T \}$ and for all $u_t \in U_t$.
\end{enumerate}
From $(A6)$, there exist $s\in \N$ such that the set $A_s:=D_2f_s(\hat{x}_s,\hat{u}_s)(T_{U_s}(\hat{u}_s))$ contains a closed convex subset $K$ with $\textnormal{ri}(K)\neq \emptyset$ and $\overline{\textnormal{Aff}(K)}$ is of finite codimension in $X$. Since the set of the lists of multipliers of a maximization problem is a cone, using the above consequences of Lemma \ref{lem45}, we can normalize the pair $(\lambda^{T}_0,p_s^{T})\neq (0,0)$ and so we can assume that $\lambda^{T}_0 + \Vert p_s^{T} \Vert_{X^*} = 1$. By combining Lemma \ref{lem47} and Proposition \ref{prop39} applied with $K$, we get a strictly increasing map $k\mapsto T_k$, from $\N$ into $\N$, and $\lambda_0\in \R^+$ and $(p_t)_{t\geq 1}\in (X^*)^{\N}$ such that: 
\begin{itemize}
\item[(i)] $\lambda^{T_k}_0\longrightarrow \lambda_0\geq 0$ when $k\rightarrow +\infty$,
\item[(ii)] for each $t \in \N$, $p^{T_k}_t \stackrel{w^*}{\longrightarrow} p_t$ when $k\rightarrow +\infty$,
\item[(iii)] $(\lambda_0,p_s)\neq (0,0).$
\end{itemize}
Thus, by doing $k \rightarrow + \infty$ in $(c)$ and $(d)$ we obtain $(3)$ and $(4)$. From $(i)$ we get $(2)$. Now, if there exists $t> s$ such that $(\lambda_0, p_t)=(0,0)$, we proceed recursively using $(3)$ to obtain that $(\lambda_0, p_s)=(0,0)$ which is a contradiction with $(iii)$. Thus, for all $t\geq s$, $(\lambda_0, p_t)\neq (0,0)$ this gives the part $(1)$. 
\end{proof}
\begin{proof}[\bf Proof of Theorem \ref{thm23}] We proceed as in the proof of Theorem \ref{thm22}, replacing the use of Lemma \ref{lem45} by the use of Lemma \ref{lem46}.
\end{proof}
%%%%%%%%%%%%%%%%%%%%%%%%%%%%%%%%%%%%%%
%%%%%%%%%%%%%%%%%%%%%%%%%%%%%%%%%%%%%%%
%%%%%%%%%%%%%%%%%%%%%%%%%%%%%%%%%%%%%%%
\section{Appendix: Some additional applications.} \label{Section 6}
%%%%%%%%%%%%%%%%%%%%%%%%%%%%%%%%%%%%%%%
%%%%%%%%%%%%%%%%%%%%%%%%%%%%%%%%%%%%%%%
%%%%%%%%%%%%%%%%%%%%%%%%%%%%%%%%%%%%%%%
In this section we establish some additional consequences of the abstract result (Lemma \ref{lem33}). We begin by the following extension of [Theorem 2.5.4 \cite{HP}]. The [Theorem 2.5.4 \cite{HP}] can be obtained by taking $K=Z$, $a=0$ and $B=B_Z(0,1)$ in Proposition \ref{prop61}.
\begin{proposition}\label{prop61} Let $Z$ be a Banach space, $\mathcal{T}$ be any nonempty set and $(p_n)_{n\in \mathcal{T}}$ be a collection of lower semicontinuous and subadditive functions from $Z$ into $\R$. Let $K$ be a closed convex subset of $Z$ such that $\textnormal{ri}(K)\neq \emptyset$. Suppose that for each $z\in K$ we have $\sup_{n\in \mathcal{T}} p_n(z) < +\infty$. Then, for each $a\in K$ and each bounded subset $B$ of $\overline{\textnormal{Aff}(K)}$, we have 
$$\sup_{n\in \mathcal{T}} \sup_{z\in B} p_n(z-a)< +\infty.$$
\end{proposition}
\begin{proof} The proof is immediat by using Lemma \ref{lem33} with $\lambda_n=1$ for all $n\in \mathcal{T}$.
\end{proof}
\vskip3mm
The above proposition is in fact an extention to subadditive functions of the classical Banach-Steinauss theorem. 
\begin{corollary}\label{cor62} (Banach-Steinauss) Let $X$ be a Banach space and $Y$ be a normed vector space. Let $\mathcal{T}$ be any nonempty set. Suppose that $(T_n)_{n\in \mathcal{T}}$ is a collection of continuous linear operators from $X$ to $Y$.  Suppose that for each $x\in X$ one has
$$\sup_{n\in \mathcal{T}} \|T_n(x)\|_Y < +\infty,$$ 
then
$$\sup_{n\in \mathcal{T}}\sup_{\|x\|=1} \|T_n(x)\|_Y=\sup_{n\in \mathcal{T}} \|T_n\|_{B(X,Y)} < +\infty.$$
\end{corollary}
\begin{proof} The proof follows immediately from Proposition \ref{prop61} applied with: $K=X$, $a=0$, the bounded set $S_X$ (the unit sphere of $X$) and with the collection of the continuous subadditive functions $p_n(x):=\|T_n(x)\|_Y$ for all $n\in \mathcal{T}$ and all $x\in X$.  
\end{proof}
We also have the following corollary.
\begin{corollary}\label{cor63} Let $A$ be a nonempty set and $(Z,\|.\|)$ be a Banach space. Let $\varphi: A\times Z \longrightarrow \R$ be a map such that:
\begin{itemize}
\item[(1)] For all $x\in A$, the map $z\mapsto \varphi(x,z)$ is lower semicontinuous and sublinear.
\item[(2)] For all $z\in Z$, the map $x\mapsto \varphi(x,z)$ is bounded.
\end{itemize}
\vskip2mm
Then, there exists a real number $C\in \R$ such that $\sup_{x\in A}\varphi(x,z)\leq C \|z\|$, for all $z\in Z$.
\end{corollary}
\begin{proof} We apply Proposition \ref{prop61} with $\mathcal{T}=A$, $p_x:=\varphi(x,.)$, using $(1)$ and $(2)$, there exists $C\in \R$ such that 
$\sup_{x\in A}\sup_{\|z\|=1} \varphi(x,z)\leq C.$
Thus, by the homogeneity of $p_x$, we have $\varphi(x,z)\leq C \|z\|$, for all $x\in A$ and all $z\in Z$.
\end{proof}
Finally, we get the following proposition, which gives, a necessary and sufficient condition such that the Dirac masses are continuous functionals.
\begin{proposition}\label{prop64}
Let $X$ be a nonempty set and $(\mathcal{B}(X),\|.\|_{\infty})$ be the Banach space of all bounded real-valued functions. Let $Y\subset \mathcal{B}(X)$ be a subspace and $\|.\|_Y$ be a norm on $Y$ such that $(Y,\|.\|_Y)$ is a Banach space. Let us denote by $\delta_x$ the Dirac mass or the evaluation at $x\in X$ defined by $\delta_x : f\mapsto f(x)$ for all $f\in \mathcal{B}(X)$. Then, the following assertions are equivalent. 

$(a)$ $\delta_x : (Y,\|.\|_Y)\longrightarrow \R$ is continuous for each $x\in X$, 

$(b)$ there exists a constant $\alpha\in \R^{+*}$ such that $\|.\|_Y\geq \alpha \|.\|_{\infty}$.
\end{proposition}
\begin{proof}
Indeed, suppose that $\delta_x : (Y,\|.\|_Y) \longrightarrow \R$ is continuous for each $x \in X$. Consider the map $\varphi: X \times Y \longrightarrow \R$ defined by $\varphi(x,f)=f(x)$ for all $(x, f)\in X\times Y$. This map satisfies the hypothesis of Corollary \ref{cor63}, so there exists $C \in \R$ such that $\sup_{x\in X} f(x)=\sup_{x\in X} \varphi(x,f)\leq C\|f\|_Y$ for all $f\in Y$. Thus by symmetry, $\sup_{x\in X} |f(x)|=\|f\|_{\infty}\leq C\|f\|_Y$ for all $f\in Y$. This implies that $C>0$ and so we take $\alpha:=\frac 1 C$. For the converse, we have $|\delta_x(f)|=|f(x)| \leq \|f\|_{\infty}\leq \frac 1 \alpha \|f\|_Y$ which shows that $\delta_x$ is continuous on $(Y,\|.\|_Y)$ since it is linear.
\end{proof}
\bibliographystyle{amsplain}

\end{document}